\documentclass[11pt]{article}
\usepackage{amssymb}
\usepackage{graphicx}
\usepackage{amsmath}
\usepackage{multirow}
\usepackage{caption}
\usepackage{subcaption}
\usepackage{enumerate}
\usepackage{mathrsfs}
\usepackage[labelfont=bf]{caption}
\usepackage[usenames]{color}

\newcommand{\R}{\mathbb{R}}
\newcommand{\Z}{\mathbb{Z}}
\newcommand{\N}{\mathbb{N}}

\DeclareMathOperator{\II}{II}

\DeclareMathOperator{\Hess}{Hess}

\DeclareMathOperator{\Ind}{Ind}

\newtheorem{theorem}{Theorem}[section]
\newtheorem{proposition}{Proposition}[section]
\newtheorem{lemma}{Lemma}[section]
\newtheorem{remark}{Remark}[section]

\newtheorem{definition}{Definition}[section]

\newenvironment{proof}[1][Proof]{\textbf{#1.} }{\ \rule{0.5em}{0.5em}}

\begin{document}
\title{On an Arnold's Conjecture Concerning the Space of Hyperbolic Homogeneous Polynomials}

\author{Vinicio A. Gómez-Gutiérrez\thanks{Facultad de Ciencias, Universidad 
		Nacional Aut\'onoma de M\'exico, C.U.,
		 CDMX 04510, M\'exico. \newline e-mail: vgomez@ciencias.unam.mx} \ 
	and Adriana Ortiz-Rodr\'iguez\thanks{Instituto de Matem\'aticas, Universidad 
		Nacional Aut\'onoma de M\'exico. C.U., CDMX 04510, M\'exico. \newline 
e-mail: aortiz@matem.unam.mx} 
}
\date{}
\maketitle

\hfill{{\it To the memory of V. I. Arnold}\quad}

\begin{abstract}
\noindent The set of homogeneous polynomials of degree $D$ is a topological space that contains  
the subset $Hyp(D)$ constituted only by hyperbolic polynomials. In 2002, V. I. Arnold conjectured
in \cite{arn0} that the number of connected components of $Hyp (D)$ increases, as $D$ increases, at 
least as a linear function of $D$. In this paper we prove that this conjecture is true. We determine the 
exact number of connected components of $Hyp (D)$ and we provide a representative for each component.  
The proof is constructive; our approach uses homotopy invariance of the index of a curve and properties of 
homogeneous polynomials. We also describe some geometrical properties of the hyperbolic 
polynomials that we provide.
\end{abstract}\medskip

\noindent {\small{\bf Keywords}: Hyperbolic homogeneous polynomials; Hessian topology; Real Quadratic Forms; Asymptotic fields of lines.}

\noindent {\small {\bf Mathematics Subject classification}: 53A05, 26C05, 54C30}

\vspace{0.1in}

\section{Introduction}

Let us consider  $H_D[x,y]$ as the set of homogeneous polynomials 
$f(x,y)=a_0 x^D + a_1 x^{D-1}y +\cdots + a_D y^D$ of degree $D \geq 2$ in $\R[x,y]$.
We say that $f$ is hyperbolic if its Hessian polynomial $\Hess f := f_{xx}f_{yy}-f_{xy}^2$ is negative at any point in the $xy$-plane other than the origin. Let $Hyp(D)$ be the set constituted by all hyperbolic homogeneous  polynomials of degree $D$ in $\R[x,y]$.
The topological properties of $Hyp(D)$ have been studied as a part of the topic named by V. I. Arnold as Hessian Topology \cite{arn0}, \cite[problems 2000-1, 2000-2, 2001-1, 2002-1]{arn3}, which studies 
the topological properties of the parabolic curve of a smooth surface in the three-dimensional Euclidean space, see for instance \cite{segre, F, B-T, K-T, arn1, panov1, arn2}. The parabolic curve of a smooth surface $S$ is made up of the points in $S$, where the second fundamental form of $S$ is degenerate.
\medskip

In particular, the Hessian topology studies the topology of the parabolic curve when the surface is the graph of a real polynomial in two variables. This problem can be divided into two questions: one of these consists of determining the maximum number of components of the parabolic curve, and the second, in establishing the mutual distribution of the components. For instance, if the parabolic curve of a polynomial of degree five is compact, the 
maximum number of components it can have is, according to Harnack's Theorem, eleven, but no polynomial of degree five with eleven compact parabolic components is known until now. 
Papers studying the topology of the parabolic curve of the graph of a polynomial include among others: 
\cite{O-S, B-B, H-O-S, H-O-S-1, C-O, K-U}.
The behavior of the parabolic curve at infinity is determined by the leading homogeneous part of the 
polynomial defining the surface.
For instance, if the leading homogeneous part is hyperbolic or elliptic, then the parabolic curve is 
compact \cite[Theorem 3.5]{GuOr}. 
This is one of the motivations that led us to be interested in the analysis of hyperbolic homogeneous polynomials.
\medskip

The connectedness property of the subset $Hyp(D)$ was investigated in \cite{arn0} 
for $D \geq 3$; namely the author proves that $Hyp(3)$ and $Hyp(4)$ are connected sets, while  
$Hyp(D)$ is disconnected whenever $D \geq 5$. In connection with this property, Arnold states the 
following conjecture \cite[p.1067]{arn0}, \cite[p.139]{arn3}:
 
\begin{description}
\item[ ] {\it \qquad\ ``The number of connected components of the space of hyperbolic homogeneous \newline
polynomials of degree $D$ increases as $D$ increases 
 {\rm (}at least as a linear function of $D${\rm)}."\qquad \ }
\end{description}

\medskip  

We recall that the second fundamental form of a smooth surface of the form 
$z - f(x,y)=0$, where $f$ is a differentiable function defined on the $xy$-plane, at a point 
$(x,y,f(x,y))$ is defined as $\II_f (x,y) = f_{xx}(x,y) dx^2 + 2 f_{xy}(x,y) dx dy+ f_{yy}(x,y) dy^2$. 
 To each hyperbolic homogeneous polynomial $f\in \R[x,y]$, a continuous map
$\gamma_f  : [0,2\pi] \rightarrow \R^3$ from the interval $[0,2\pi]$
into the three-dimensional space of the real quadratic forms on the real plane 
(i.e., $\{ a\xi^2 + 2b \xi \eta + c\eta^2\}$) is associated \cite{arn0}. This map assigns to each $\varphi$ 
the second fundamental form of the graph of $f$ at the point $(x,y) = (\cos \,\varphi, sin \,\varphi)$. 
The index of $\gamma_f$, 
$\Ind (\gamma_f)$, is the number of revolutions that the point 
$\gamma_f (\varphi)$ makes around the cone of de\-ge\-nerated quadratic forms $\{b^2 - ac = 0\}$ 
with $\varphi \in [0,2\pi]$. \medskip

The question of when a homogeneous polynomial is hyperbolic or not, was analyzed 
in polar coordinates in \cite{arn0}. A consequence of this study \cite[Theorem 2 $\&$ Theorem 4 of \S 3]{arn0} is that for any hyperbolic homogeneous polynomial $f$:
\begin{equation}\label{indice-polares}
\Ind (\gamma_f) = 2 - \frac{1}{2} \# \{\varphi \in [0,2\pi) : F(\varphi) = 0\},
\end{equation}
where $F$ is the restriction of $f$ to the unit circle, that is, $F (\varphi) = f(\cos \,\varphi , \sin \,\varphi )$. 
Denoting by $D$ the degree of $f$ we obtain from equality (\ref{indice-polares}) that (see inequalities (\ref{cota-ind-gama})):
$$ 2 - D \leq \Ind (\gamma_f) \leq 2.$$
On the other hand, it is proved in \cite[Corollary of \S 6 $\&$ Corollary 4 of \S 7]{arn0} that when the number
$D$ is odd, $\Ind (\gamma_{f}) \leq -1$ while in the case of $D$ even, $\Ind (\gamma_{f}) \leq 0$. Moreover,
for both cases, $\Ind (\gamma_{f}) \equiv D$ mod ($2$) is fulfilled. 
From the previous results, the only allowed values for $\gamma_f$ are the real
numbers in the set $\{ 0, -2, -4 , \cdots , -(D-2) \}$ if $D$ is even and $D\geq 6$, while for
each odd number $D\geq 3,$ the allowed values are $\{ -1, -3, -5 \cdots , -(D-2) \}$. That is, 
there are exactly $\frac{D}{2}$ allowed values for the index of $\gamma_f$ in the even case 
and $\frac{D-1}{2}$ for the odd case. Later we will show that there are hyperbolic polynomials with the property that the indices of their associated curves $\gamma_f$  reach these values.

\medskip

In connection with the aforementioned conjecture, Arnold proves \cite[Theorem 1 of 
 $\S$5]{arn0} 
that the homogeneous polynomials of degree $D$:
\begin{eqnarray}\label{polis-arnold}
 f_m (x,y) = (x^2+y^2)^{\frac{D-m}{2}} \mbox{Re} (x+i y)^m 
\end{eqnarray}
are hyperbolic for every  natural number $m$ satisfying the relations
$\, m\leq D < m^2 \,$ and $\,D-m\,$ is even. He also 
obtained that the Poincar\'e index of the field of asymptotic directions of each $f_m$ at the 
origin $\Ind_0 (\II_{f_m})$ satisfies $ \Ind_0 (\II_{f_m}) = \frac{2-m}{2}$. Since 
$ \Ind_0 (\II_{f_m}) = \frac{1}{2}\Ind (\gamma_f)$ \cite[Theorem 2 of \S 2]{arn0}, then
$\Ind (\gamma_{f_m}) = 2-m$.  Taking into account 
the notation $P_m (x,y) := 
\mbox{Re} (x+i y)^m$ and $Q_{D-m} (x,y) :=  (x^2+y^2)^{\frac{D-m}{2}}$, we show in Table 
\ref{tabla-1} the values ($f$, $\Ind (\gamma_f)),$ where $f = P_m Q_{D-m}$ is a  hyperbolic polynomial 
 in the family (\ref{polis-arnold}) of degree $D$ whenever $3 \leq D\leq 16$, and $\Ind (\gamma_f)$ is the
index of the curve $\gamma_f$. When $f = P_m Q_0,$ we will omit $Q_0$.

\medskip 
\begin{table}[h!] 
\begin{center}{\small
\begin{tabular}{|l| l l l l l  l |}\hline 
$D$  & ($f$, $\Ind (\gamma_f))$& & & & & \\ \hline
2  & $(P_2, 0)$ & & & & &\\ \hline
3  & $(P_3, -1)$ & & & & &\\ \hline
4  & $(P_4, -2)$ & & & & &\\ \hline
5  & $(P_5, -3)$ & $(P_3 Q_2, -1)$ & & & & \\ \hline
6  & $(P_6, -4)$ & $(P_4 Q_2, -2)$ & & & & \\ \hline
7  & $(P_7, -5)$ & $(P_5 Q_2, -3)$ & $(P_3 Q_4, -1)$ & & & \\ \hline
8  & $(P_8, -6)$ & $(P_6 Q_2, -4)$ & $(P_4 Q_4, -2)$ & & & \\ \hline
9  & $(P_9, -7)$ & $(P_7 Q_2, -5)$ & $(P_5 Q_4, -3)$ & & &\\ \hline
10  & $(P_{10}, -8)$ & $(P_8 Q_2, -6)$ & $(P_6 Q_4, -4)$ & $(P_4 Q_6, -2)$ & &\\ \hline
11  & $(P_{11}, -9)$ & $(P_9 Q_2, -7)$ & $(P_7 Q_4, -5)$ & $(P_5 Q_6, -3)$ & &\\ \hline
12  & $(P_{12}, -10)$ & $(P_{10} Q_2, -8)$ & $(P_8 Q_4, -6)$ & $(P_6 Q_6, -4)$ & 
$(P_4 Q_8, -2)$ &\\ \hline
13  & $(P_{13}, -11)$ & $(P_{11} Q_2, -9)$ & $(P_9 Q_4, -7)$ & $(P_7 Q_6, -5)$ &
$(P_5 Q_8, -3)$ &\\ \hline
14  & $(P_{14}, -12)$ & $(P_{12} Q_2, -10)$ & $(P_{10} Q_4, -8)$ & $(P_8 Q_6, -6)$ &
$(P_6 Q_8, -4)$ & $(P_4 Q_{10}, -2)$ \\ \hline
15  & $(P_{15}, -13)$ & $(P_{13} Q_2, -11)$ &  $(P_{11} Q_4, -9)$ &  $(P_9 Q_6, -7)$ &
$(P_7 Q_8, -5)$ & $(P_5 Q_{10}, -3)$\\ \hline
16  & $(P_{16}, -14)$  & $(P_{14} Q_2, -12)$  & $(P_{12} Q_4, -10)$  & $(P_{10} Q_6, -8)$  
&  $(P_8 Q_8, -6)$  & $(P_6 Q_{10}, -4)$ \\ \hline
\end{tabular}\caption{Hyperbolic polynomials of degree $D$ in the Arnold family 
(\ref{polis-arnold}).} \label{tabla-1}}
\end{center}
\end{table}

We remark from Table \ref{tabla-1} that for any odd number $D$ such that $D\geq 9$, 
there is no polynomial $f$ of degree $D$ in the Arnold family 
with the property $\Ind (\gamma_f) = -1$. In an analogous way, if $D\geq 25$
the values $ \{-1, -3\}$ are not reached as indices of $\gamma_f$ for any polynomial $f$ in the family
(\ref{polis-arnold}). More  generally,  for each 
natural number $k\geq 1$ and each odd number $D$ satisfying $D\geq (2k+1)^2$, the $k$ 
values that are not reached as indices of $\gamma_f$ by any polynomial $f$ of degree $D$ in the family (\ref{polis-arnold}) are $\{-1, -3, \ldots ,-(2k-1)\}$.

\medskip

In the case that $D$ is even, 
we remark that for $D\geq 6,$ the value $0$ is not 
achieved as the index of $\gamma_{f}$ by any polynomial of degree $D$ belonging to the Arnold family.
On the other hand, no polynomial $f$ of the family (\ref{polis-arnold}) satisfies that its associated 
curve $\gamma_f$ has index $-2$ if $D \geq 16$. In a similar way to the odd case, we conclude that for 
each natural number $k\geq 2$ and each even number $D$ fulfilling  $D\geq (2k)^2$, 
the $k$ values that are not reached as indices of $\gamma_f$ by any polynomial 
$f$ of degree $D$ belonging to the family (\ref{polis-arnold}) are $\{0, -2, \ldots ,-(2k-2)\}$.
\medskip

In this paper we prove Arnold's conjecture, Theorem \ref{teo-conjetura}. On the one hand,
we exhibit a family $\mathcal{F}$ of hyperbolic polynomials, different from Arnold's 
family, such that for each odd number $D\geq3$ there are $\frac{D-1}{2}$ hyperbolic polynomials $f \in
\mathcal{F}$ with the property that their associated hyperbolic curves $\gamma_f$ have different
values, and for each even number $D$ there are $\frac{D}{2}$ hyperbolic polynomials 
in $\mathcal{F}$ with the same property. 
In both cases, indices of curves $\gamma_f$ for polynomials $f$ in our family attain all the possible 
values that $\gamma_f$ could take.
\medskip

On the other hand, we prove in Theorem \ref{poli-homotopos} that two hyperbolic polynomials $f, g \in H_D[x,y]$ of degree $D$ are in the same connected 
component of $Hyp(D)$ if and only if the indices of their corresponding hyperbolic curves 
$\gamma_f$ and $\gamma_g$ are equal. From this theorem, it follows that $Hyp(D)$ has exactly 
$\frac{D-1}{2}$ connected components for any odd $D\geq 3$, there are exactly $\frac{D}{2}$ for $D\geq 6$ even, and in both
cases each of these connected components contains a polynomial of our family.

\medskip In subsection 2.2, we exhibit some geometrical properties of the hyperbolic 
polynomials of the family $\mathcal{F}$.


\section{Results and Proofs}
\subsection{Arnold's Conjecture}
The set of real homogeneous polynomials of degree $D$ in two variables,  $H_D[x,y]$, is a space 
homeomorphic to $\R^{D+1}\setminus \{0\}$. We consider this set equipped with the induced 
topology of the Euclidean space $\R^{D+1}$. 
This topological space contains {\it hyperbolic polynomials} $f$ defined by the condition that their hessian polynomial $\Hess f := f_{xx}f_{yy}-f_{xy}^2$ is negative in $\R^2\setminus \{(0,0)\}$.
We are interested in studying the topology of the subset $Hyp (D) \subset H_D[x,y]$, which consists only 
of hyperbolic polynomials. Particularly, we will focus on the study of the quantitative analysis of the 
connected components of $Hyp (D)$. 

\begin{theorem}[Arnold's conjecture]\label{teo-conjetura}
The subset $Hyp (D)$ has exactly $\frac{D-1}{2}$ connected components if $D$ 
is odd; while for each even $D\geq 6$, it has exactly $\frac{D}{2}$
connected components.
\end{theorem}

\noindent\begin{proof}
The proof is divided into three steps. The first one consists of verifying that two polynomials $f,g \in
Hyp (D)$ are in the same connected component if and only if their associated curves $\gamma_f, \gamma_g$
(which will be defined below) have equal indices. \medskip

In the second step we justify that $Hyp(D)$ has at most $\frac{D-1}{2}$ connected components if $D$ 
is odd; and $\frac{D}{2}$ whenever $D$ is even.\medskip

The third step is to construct, for each odd number $D$, $\frac{D-1}{2}$ hyperbolic 
polynomials $f$ of degree $D$ with the property that their hyperbolic curves $\gamma_f$ have different 
indices. Therefore, according to the first step, these hyperbolic polynomials will be in different connected 
components of $Hyp(D)$; and by the second step, $Hyp(D)$ will have exactly $\frac{D-1}{2}$ connected 
components.
Carrying out an analogous reasoning for each even number $D$, we construct $\frac{D}{2}$ hyperbolic 
polynomials of degree $D$ such that their hyperbolic curves have different indices. \medskip

In order to be more general, we decided to isolate the results, developed at each step, from their original 
context. 
\end{proof}  
\medskip

\noindent {\bf First step.}
\begin{definition}
{\rm Let $\gamma : [0,2\pi] \rightarrow \mathcal{Q}$ be a continuous curve from the interval 
$[0,2\pi]$ into the three-dimensional space of real quadratic forms on the plane $\mathcal{Q} =
\{ a\psi^2 + 2 b\psi \eta + c\eta^2 : a,b,c \in \R\}$.  We say that} $\gamma$ is 
hyperbolic if for each $\varphi \in [0, 2\pi]$ {\rm the quadratic form $\gamma 
(\varphi)$ has the 
signature (+,$-$), that is, if 
$$ a (\varphi) c (\varphi) - b^2 (\varphi) < 0.$$ }
\end{definition}

Let $\gamma : [0,2\pi] \rightarrow \mathcal{Q}$ be a hyperbolic curve. We define {\it  the index of  $\gamma$} 
as the number of revolutions of $\gamma$ about the cone of degenerate forms $\{ ac - b^2 = 0\}$, that is,
\medskip

\centerline{$\, \Ind (\gamma) = \frac{1}{2\pi} \Delta$ arg ($a-c + 2bi$)($\varphi$)$|_0^{2\pi} \in \Z.$}
\medskip

It is well known that,  for a hyperbolic homogeneous polynomial $f$, its second fundamental form 
$\II_f (x,y) = f_{xx}(x,y) dx^2 + 2 f_{xy}(x,y) dx dy+ f_{yy}(x,y) dy^2$  defines two continuous 
line fields in $\mathbb R^{2^*} = \R^ 2\setminus \{(0,0)\}$. These line fields are obtained
by projecting the fields of asymptotic lines into the $xy$-plane under the orthogonal 
projection $\pi:\R^3 \rightarrow \R^2, (x,y,z)\mapsto (x,y)$. We identify, 
for the sake of simplicity, the zero locus of $\II_f$ with the asymptotic directions and call them 
{\it the field of asymptotic directions of $f$.} Thus, each hyperbolic homogeneous polynomial 
$f$ is associated with a closed hyperbolic continuous curve $\gamma_f : [0,2\pi] \rightarrow 
\mathcal{Q}$ such that $\varphi \mapsto \II_f (cos \,\varphi, sin \,\varphi)$, that is, 
$$\gamma_f (\varphi) := 
(f_{xx}(cos \,\varphi, sin \,\varphi),\, f_{xy}(cos \,\varphi, sin \,\varphi), \,
f_{yy}(cos \,\varphi, sin \,\varphi)).$$

In what follows we will see that the index of $\gamma_f$ plays an important role in determining 
the number of connected components of $Hyp (D)$. 

\begin{definition}
{\rm We define} a D-hyperbolic isotopy of two hyperbolic 
polynomials $f, g \in Hyp (D)$ {\rm as a continuous function $H:[0,1]\times \R^2 \rightarrow \R$ 
such that $\, H(0,x,y) = f(x,y), \, H(1,x,y) = g(x,y)$ and
for any $t \in (0,1), \, H(t, x, y)$ is a hyperbolic polynomial in $Hyp (D)$ evaluated in $(x,y)$. 
In this case, we say that $f$ and $g$ are} $D$-hyperbolic isotopic. 
\end{definition}

A D-hyperbolic isotopy of two hyperbolic polynomials $f, g \in Hyp (D)$ is essentially a continuous curve
from the interval $[0,1]$ into the subset $Hyp (D)$ that connects $f$ and $g$. 

\begin{remark}
The set $Hyp (D)$ is an open set in $\R^{D+1}$.
\end{remark}

\noindent \begin{proof}
Consider the continuous functions, $H: H_D [x,y] \rightarrow H_{2D-4}[x,y]\cup \{\bar 0\}$ such that to each
homogeneous polynomial $f$ of degree $D$ assigns its Hessian polynomial $ f_{xx}f_{yy} - f_{xy}^2$; 
and the function $M: H_{2D-4}[x,y]\cup \{\bar 0\} \rightarrow \R$ such that to each polynomial $g$ assigns
the maximum value of the restriction of $g$ to the unit circle. 
The statement follows since $Hyp (D)$ is the inverse image of the 
interval $(- \infty , 0)$ under the continuous function $M\circ H$.
\end{proof}\medskip

Since $Hyp (D)$ is an open subset of $\R^{D+1}$, 
two hyperbolic polynomials in $Hyp (D)$ are $D$-hyperbolic isotopic if and only if 
they lie in the same connected component of $Hyp (D)$.

\begin{theorem}\label{poli-homotopos}
Two hyperbolic polynomials $f, g \in Hyp (D)$ are $D$-hyperbolic isotopic
if and only if $\Ind (\gamma_f) = \Ind (\gamma_g)$.
\end{theorem}

\noindent \begin{proof}
Let us suppose that $H$ is a $D$-hyperbolic isotopy of the polynomials $f$ and $g$. In order to prove 
that the indices of $\gamma_f$ and $\gamma_g$ are equal, we will prove that these curves are homotopic.
An explicit expression for $H$ is $H(t,x,y) = \sum_{i=0}^D u_i(t) x^{D-i} y^i$, where
$u_i : [0,1] \rightarrow \R$ is a continuous function for $i\in \{0, \ldots ,D\}$.
For each $t\in [0,1],$ denote by $Q_t(x,y)$ the hyperbolic polynomial $H(t,x,y)$. The associated hyperbolic 
curve to $Q_t$ is
$$\gamma_{Q_t} (\varphi) = \left(\frac{\partial^2 H}{\partial x^2}(t, cos \,\varphi, sin \,\varphi),
\ \frac{\partial^2 H}{\partial x \partial y}(t, cos \,\varphi, sin \,\varphi),
\ \frac{\partial^2 H}{\partial y^2}(t, cos \,\varphi, sin \,\varphi)\right),$$
that is,
\begin{align*}
\,\gamma_{Q_t} (\varphi) = &\left( \sum_{i=0}^{D-2} (D-i) (D-i-1) u_i(t) (cos \,\varphi)^{D-i-2} (sin \,\varphi)^i, 
\right. \\
&\left. \ \ \, \sum_{i=1}^{D-1} i (D-i) u_i(t) (cos \,\varphi)^{D-i-1} (sin \,\varphi)^{i-1}, \
 \sum_{i=2}^{D} i (i-1) u_i(t) (cos \,\varphi)^{D-i} (sin \,\varphi)^{i-2}
\right).
\end{align*}

So, the homotopy $\Gamma$ we are looking for the hyperbolic curves $\gamma_f$ and 
$\gamma_g$ is the continuous function $\Gamma :[0,1]\times [0,2\pi] \rightarrow 
\mathcal{Q}$ defined as $\Gamma(t,\varphi) := \gamma_{Q_t} (\varphi)$.

\medskip
Now, suppose that $\Ind (\gamma_f) = \Ind (\gamma_g)$. Consider the Euclidean space 
$\R^3$ with coordinates $(u,v,w)$, and the subset described by:
\begin{equation}\label{set-hyper-cond}
D^2 u^2 + Duw - (D-1)v^2 < 0.
\end{equation}
This subset is a connected component of the complement of 
the cone $D^2 u^2 + Duw - (D-1)v^2 = 0$ that is homotopy equivalent to the circle.
On the other hand, in \cite[Theorem 1]{arn0} the author shows that: 

\centerline{\it A polynomial 
$f$ is $D$-hyperbolic if and only if the curve $F(\varphi) = 
f(cos \,\varphi, sin \,\varphi)$} 
\centerline{\it satisfies the hyperbolicity condition}
\begin{equation}\label{hyper-cond}
D^2 F^2 + D F F\, '' - (D-1) (F\, ')^2 <0.
\end{equation}
 
\noindent It follows that $F$ can have only simple zeros and nondegenerate critical points on the circle.
For our polynomial $f$ (and doing the same for $g$) we define in $\R^3 = \{(u,v,w)\}$ 
the curve
\begin{equation}
\alpha_f (\varphi) := \big(F(\varphi), F\, ' (\varphi), F\, ''(\varphi)\big).
\end{equation}
Since $f$ is hyperbolic, $\alpha_f (\varphi)$ is in the subset defined in (\ref{set-hyper-cond}) 
for each $\varphi \in [0,2\pi)$.

\begin{lemma}\label{igual-idices-ga-al}
$\Ind (\gamma_f) = 2 + \Ind (\alpha_f).$
\end{lemma}

\noindent \begin{proof}
Consider the vertical projection $T$ from $\R^3 = \{(u,v,w)\}$ onto the plane $\overline{\Pi}: 
2D u + w =0$, and denote by $\overline{\alpha_f}$ the projection of $\alpha_f$ onto this plane.
Note that the plane $\overline{\Pi}$ intersects the cone $D^2 u^2 + Duw - (D-1)v^2 = 0$ only
at the origin. When we restrict the projection $T$ to the subset (\ref{set-hyper-cond}) it induces 
an isomorphism $T_*$ from the fundamental group of the subset (\ref{set-hyper-cond}) onto the 
fundamental group of the punctured plane $\overline{\Pi}\setminus \{ \bar 0 \}$. It can be shown 
that the index of  $\alpha_f$ in the subset (\ref{set-hyper-cond}) is the same as the classical index of 
its projection $T (\alpha_f) = \overline{\alpha_f}$ in the punctured plane $\overline{\Pi}\setminus 
\{\bar 0\}$.
On the other hand, the map $Z: \overline{\Pi} \rightarrow \R^2 =\{(u,v)\},\, 
Z(u,v,w) = (u,v)$ induces an isomorphism between the fundamental group of $\overline{\Pi}\setminus 
\{\bar 0\}$ and that of $\R^2 \setminus \{\bar 0\}$. That is, the number of turns that  the curve
$\overline{\alpha_f} = T \circ \alpha_f $ makes around the origin  in the plane $\overline{\Pi}$ is the same 
as the number of turns that $\zeta $ makes around the origin in $\R^2$. So,  
\begin{eqnarray}\label{indices-equival}
\Ind\, ({\alpha_f}) = \Ind\, (\overline{\alpha_f}) = \Ind\, (\zeta).
\end{eqnarray}
Under the first isomorphism, the class 
of  $\alpha_f (\varphi) = \big(F(\varphi), F\, ' (\varphi), F\, ''(\varphi)\big)$ corresponds to the 
class of $\overline{\alpha_f}(\varphi)= \big(F(\varphi), F\, ' (\varphi), -2DF (\varphi)\big)$, and under the second isomorphism, this last one corresponds to the class of $\zeta (\varphi) = 
\big(F(\varphi), F\, ' (\varphi)\big)$.
From Theorem 2 of $\S$3 of \cite{arn0}, it follows that each revolution that $\zeta$ makes around the 
origin, it intersects each coordinate axis exactly twice. Thus, for each revolution that 
$\zeta$ does, $F$ has two more critical points. That is:
\begin{eqnarray}\label{indice-proy-alfa}
\Ind\, (\zeta) = -\frac{1}{2}\# \{\varphi \in [0, 2\pi): F\, '(\varphi) = 0\}.
\end{eqnarray}

\noindent On the other hand, it is proved in \cite[Theorem 3 and Theorem 4]{arn0} that the 
index of the hyperbolic curve $\gamma_f$ is determined by the number of critical points 
of the function $\, F$, that is,
\begin{equation}\label{index-gama}
\Ind (\gamma_f) = 2 - \frac{1}{2} \# \{\varphi \in [0,2\pi) : F\, ' (\varphi) = 0\}.
\end{equation}
According to (\ref{indices-equival}), (\ref{indice-proy-alfa}) and  (\ref{index-gama}), we obtain that $\Ind (\gamma_f) = 2 + 
\Ind (\alpha_f)$.
\end{proof}
\bigskip 

We recall that we are assuming that $\Ind (\gamma_f) = \Ind (\gamma_g)$. It follows, 
by Lemma \ref{igual-idices-ga-al} that $\Ind (\alpha_f) = \Ind (\alpha_g)$. 
We shall now prove that the curves $\alpha_f$ and $\alpha_g$ are homotopic in the
set (\ref{set-hyper-cond}). To do this, consider a continuous map $\mu:[0,2\pi] \rightarrow \R^3$ 
such that its image is completely contained in the subset (\ref{set-hyper-cond}), its initial point is 
$\alpha_f (0)$, its final point is $\alpha_g (0)$, and its projection onto the plane $\overline{\Pi}: 
2D u + w =0$ does not make any complete rotation around the origin.

\medskip

Consider now the closed curve $\beta_g: [0,2\pi] \rightarrow \R^3$ which starts at the initial point of 
$\alpha_f$, is formed by the union of $\mu$, the whole $\alpha_g$ and $\mu^{-1}$, where
$\mu^{-1} (\varphi) = \mu (2\pi - \varphi)$, that is, 
$$\beta_g (\varphi) = \left\{\begin{array}{ll}
\mu(2\varphi) &\mbox{if } \varphi \in [0,\pi]  \\
\alpha_g(4\varphi - 4\pi) &\mbox{if } \varphi \in \left[\pi,\frac{3\pi}{2}\right]\\
\mu^{-1} (4\varphi - 6\pi) &\mbox{if } \varphi \in \left[\frac{3\pi}{2},2\pi\right].
\end{array}
\right.$$
By construction, the index of $\beta_g$ is equal to that of $\alpha_g$. Since the fundamental group of the subset (\ref{set-hyper-cond}) is $\Z$, the curves $\beta_g$ and $\alpha_f$ are homotopic in (\ref{set-hyper-cond}),
that is, there exists a continuous function $H:[0,1]\times [0,2\pi] \rightarrow \R^3$ such that 
$H(0,\varphi) = \alpha_f$, $\, H(1,\varphi) = \beta_g$, and for each $t\in [0,1]$,
the set $H(t,-)$ is a continuous curve lying in the subset (\ref{set-hyper-cond}).
\medskip

We define now a homotopy $\tilde{H}:[0,1]\times [0,2\pi] \rightarrow \R^3$ between the 
curves $\beta_g$ and $\alpha_g$, that is
$$ \tilde{H}(t,\varphi) = \left\{\begin{array}{ll}
\mu\big((1-t)2\varphi + 2t\pi\big) &\mbox{if } \varphi \in [0,\pi] \\
\alpha_g(4\varphi - 4\pi) &\mbox{if } \varphi \in [\pi, \frac{3\pi}{2}]\\
\mu^{-1} ((1-t) (4\varphi - 6\pi)) &\mbox{if } \varphi \in \left[\frac{3\pi}{2},2\pi\right].
\end{array}\right.$$
In conclusion, we have obtained a homotopy between the curves $\alpha_f$ and 
$\alpha_g$. By considering the projection from $\R^3 = \{(u,v,w)\}$ 
into the first coordinate,
this homotopy induces a homotopy $\tilde{h}$ between the functions $F$ and $G(\varphi)
= g(cos \,\varphi, sin \,\varphi)$ where $\tilde{h}:[0,1]\times [0,2\pi] 
\rightarrow \R, \, \tilde{h}(0,\varphi) = F(\varphi)$ and $\tilde{h}(1,\varphi) = 
G(\varphi)$.
The desired homotopy between the polynomials $f$ and $g$ is then the function
$$h:[0,1]\times [0,2\pi] \times \R_{\geq 0} \rightarrow \R, \quad
h(t,\varphi, r) = r^D \tilde{h}(t,\varphi),$$
$\mbox{ with } 
h(0,\varphi, r) = r^D F(\varphi) = f(x,y) \mbox{ and } 
h(1,\varphi, r) = r^D G(\varphi) = g(x,y).$
\end{proof}
\medskip

\noindent {\bf Second step.}
From Theorem \ref{poli-homotopos} we obtain a well defined function from the set of connected 
components of $Hyp (D)$ into the set of indices of $\gamma$ such that to each connected component
$C$ of $Hyp (D)$ the index of the hyperbolic curve of any polynomial $f\in C$ is associated.
Moreover, this function is injective. So, the maximal number of values that $\Ind 
(\gamma)$ could reach is an upper bound for the number of connected components of $Hyp(D)$.
In what follows we analyze the different values that $\Ind (\gamma)$ may take.

\begin{remark}[\cite{arn0}, p.1035]\label{obs-arnold}
Let $f \in \R[x,y]$ be a hyperbolic homogeneous polynomial and let $F(\varphi)= 
f(cos \,\varphi, sin \,\varphi)$ be
 the restriction of $f$ to the unit circle. Then, the number of zeroes of  the function $F$ in  
$[0, 2\pi]$  is equal to the number of its critical points.
\end{remark}

\begin{remark}\label{cotas-ind-gama}
Let $f \in \R[x,y]$ be a hyperbolic homogeneous polynomial of degree $D$. Then,
\begin{align*}\label{desig-remark-2.3}
2-D \leq &\,\Ind (\gamma_f) \leq 0 \ \ \mbox{ if } D \mbox{ is even}. \\
2-D \leq &\,\Ind (\gamma_f) \leq -1   \mbox{ if } D \mbox{ is odd}.
\end{align*}
\end{remark}

\noindent \begin{proof}
First, let us note that the equation 
(\ref{index-gama}) becomes, by remark \ref{obs-arnold}, 
\begin{equation}\label{index-gam-ceros}
 \Ind (\gamma_f) = 2 - \frac{1}{2} \# \{\varphi \in [0,2\pi) : F (\varphi) = 0\}.
\end{equation}
As the function $F (\varphi) = f (cos \,\varphi, sin \,\varphi)$ is the restriction of 
$f$ to the unit circle, then only the real linear factors of $f$ contribute to 
the set of zeroes of $F$. Moreover, each real linear factor of $f$ gives 
rise to exactly two zeroes of $F$. Then, if the degree of $f$ is $D$, $F$ has at most $2D$ 
zeroes, and
\begin{equation}\label{cota-ind-gama}
2-D \leq \Ind (\gamma_f) \leq 2.
\end{equation}
In \cite[Corollary of \S 6 $\&$ Corollary 4 of \S 7]{arn0} it is proved that if $D$ is odd, 
$\Ind (\gamma_{f}) \leq -1$; while in the even case for $D$, $\Ind (\gamma_{f}) \leq 0$. 
\end{proof}

\begin{remark}\label{pol-ind-curvas}
Two hyperbolic homogeneous polynomials in $\R[x,y]$ have the same number of real 
linear factors if and only if the indices of their associated hyperbolic curves are equal.
\end{remark}

\noindent \begin{proof}
Suppose that $f$ and $g$ have exactly $m$ real linear factors. By equation 
(\ref{index-gam-ceros}) we have that 
$\Ind (\gamma_f) = 2 -m = \Ind (\gamma_g)$ since the functions $ F (\varphi) =  f(cos \,
\varphi, sin \,\varphi)$ and 
$ G (\varphi) =  g(cos \,\varphi, sin \,\varphi)$ have exactly $2m$ zeroes. 

Inversely, suppose that $\Ind (\gamma_f) =  \Ind (\gamma_g)$. Then, by equation 
(\ref{index-gam-ceros})
the functions $F$ and $G$ have the same number of zeroes, which implies that the 
polynomials $f$ and $g$ have the same number of real linear factors.
\end{proof}
\medskip 

By the inequalities of Remark \ref{cotas-ind-gama}, and by the relation $\Ind (\gamma_{f}) 
\equiv D$ mod ($2$) for any $D$ (see results in \cite[\S 6]{arn0}), we 
conclude that when $D$ is even, the index of $\gamma_{f} $ 
may reach the values of the set $\{ 0, -2, -4 , \cdots , -(D-2) \}$. That is, there are at most
$\frac{D}{2}$ possible values for the index of $\gamma_f$ whenever $D$ is even. 
A similar analysis for the odd case leads us to claim that there are at most $\frac{D-1}{2}$ 
possible values for the index of $\gamma_f$ which are $\{ -1, -3, -5 , \cdots , -(D-2) \}$.
\medskip 

Hence, $Hyp(D)$ has at most $\frac{D}{2}$ connected components if $D\geq 6$ is even, and at most
$\frac{D-1}{2}$  for each odd number $D\geq 5$. \medskip

\noindent {\bf Third step.} 
At this stage we will exhibit a family of hyperbolic polynomials different 
from those given in (\ref{polis-arnold}) with the property that for each integer number $j$ 
allowed by the inequalities of Remark \ref{cotas-ind-gama} there exists a polynomial $f$ in this 
new family such that the index of its associate curve $\gamma_f$ equals $j$.

\begin{remark}\label{pol-fact}
For each natural number $k$ the polynomials 
$$ R_{2k+1}(x,y):= x\,\Pi_{i=1}^k (x^2 - i^2 y^2) \mbox{ \ and \ } 
 R_{2k+2}(x,y):=x \big(x- (k+1) y\big) \Pi_{i=1}^k (x^2 - i^2 y^2)$$
of degree $2k+1$ and $2k+2$, respectively, are hyperbolic. Indeed, It is proved in \cite{mgua} that any homogeneous polynomial of degree $D\geq 2$ having $D$ non-proportional real linear factors, 
is hyperbolic. Moreover, by expression (\ref{index-gam-ceros}), the index of  
$\gamma_{R_{2k+j}}$ is $2-2k-j$ for $j\in \{1,2\}$. The arguments used in \cite{mgua} to prove 
this property are the same as those developed in \cite{AOR} for generic nonhomogeneous polynomials.
\end{remark}

\begin{theorem}\label{teo-pol-hiper-impares} 
Let $n \in \N$ be a fixed number. For each $k\in \N$,
the homogeneous polynomials
\begin{align*}
&i\mbox{)}\, g_{2n+2}(x,y) :=  (x^2 - y^2) (x^{2n} + y^{2n})  \mbox{ of even degree } 2n+2, 
\mbox{ with }n\geq 2,\\
&ii\mbox{)}\, f_{2k+1}(x,y) := x\, (x^{2n} + y^{2n}) \Pi_{i=1}^k (x^2 - i^2 y^2)
\mbox{ of odd degree } 2n+2k+1, \mbox{ and}\\
&iii\mbox{)}\,  f_{2k+2}(x,y) := x\, \big(x^{2n} + y^{2n}\big)  \big(x- (k+1) y\big)
\Pi_{i=1}^k (x^2 - i^2 y^2) \mbox{ of even degree } 2n+2k+2,
\end{align*}
 are hyperbolic. Moreover, the index of $\gamma_{g_{_{2n+2}}}$ is $0$ and the index of
$\gamma_{f_{_{2k+j}}}$ is $2-2k-j $ for  $j\in \{1,2\}$.
\end{theorem}

We remark that the polynomial $g_4 (x,y) = (x^{2} - y^{2})(x^{2} + y^{2})$ (taking $n=1$)
is not hyperbolic because its Hessian polynomial $(\Hess g_4 )(x,y) = -144\,  x^{2} y^{2}$ 
equals zero along the coordinate axes.\medskip

For $D =9$ this Theorem gives the following three hyperbolic polynomials:
\begin{align*}
t_1 &= x\,  (x^{2} + y^{2})\Pi_{i=1}^3 (x^2 - i^2 y^2), &\Ind (\gamma_{_{t_1}}) = -5,\\
t_2 &= x\,  (x^{4} + y^{4})\Pi_{i=1}^2 (x^2 - i^2 y^2), &\Ind (\gamma_{_{t_2}}) = -3,\\
t_3 &= x\,  (x^{6} + y^{6}) (x^2 - y^2), &\Ind (\gamma_{_{t_3}}) = -1.
\end{align*}
Thus, for $D =9$ the following four hyperbolic
polynomials correspond to the four connected components of $Hyp (9)$: $t_1, \, t_2, \, t_3$ and 
$R_{9}(x) = x\,\Pi_{i=1}^4 (x^2 - i^2 y^2)$ with $\Ind (\gamma_{_{R_{9}}}) = -7$.
\medskip

In general, for each odd natural number $D\geq 3$, Theorem \ref{teo-pol-hiper-impares} provides 
$\frac{D-3}{2}$ different examples of hyperbolic polynomials of degree $D$ such that the indices of 
their hyperbolic curves have distinct values. Indeed, by considering the notation 
$E_{2n} (x,y) := x^{2n}+y^{2n}$ 
and that of Remark \ref{pol-fact}, we have that the $(D-3)/2$ hyperbolic polynomials of degree $D$ are:
$$E_2 R_{D-2}, \ \  E_4 R_{D-4}, \ \ldots , \ E_{D-3} R_{3},$$
and the indices of their hyperbolic curves are
$$ \Ind (\gamma_{_{E_2 R_{D-2}}}) = 4-D, \ 
\Ind (\gamma_{_{E_4 R_{D-4}}}) = 6-D, \ \ldots , \  \Ind (\gamma_{_{E_{D-3} R_{3}}}) = -1.$$
Also taking into account the hyperbolic polynomial $R_{2k+1}$, with $k =\frac{D-1}{2}$, which satisfies
that $\Ind (\gamma_{_{R_D}}) = 2-D,$ we obtain a representative polynomial of each connected component of
$Hyp (D)$. \medskip

Likewise, for each even natural number $D>4$, Theorem \ref{teo-pol-hiper-impares} provides $\frac{D}{2}-1$
hyperbolic polynomials of degree $D$ such that their hyperbolic curves have
different indices, namely,
\begin{equation}\label{hyp-grado-par}
E_2 R_{D-2}, \ \  E_4 R_{D-4}, \ \ldots , \ E_{D-4} R_{4}, \ g_{2\left(\frac{D-2}{2}\right)+2},
\end{equation}
and the indices of their hyperbolic curves are
$$ \Ind (\gamma_{_{Q_2 R_{D-2}}}) = 4-D, \,
\Ind (\gamma_{_{Q_4 R_{D-4}}}) = 6-D,\, \ldots , \,\Ind (\gamma_{_{Q_{D-4} R_{4}}}) = -2,  \,
\Ind (\gamma_{_{g_{_D}}}) =0.$$
Thus, by adding the polynomial $R_{2k+2}$ with $k =\frac{D-2}{2}$ to those shown in expression 
(\ref{hyp-grado-par}), we achieve a representative polynomial in each of the $D/2$ connected 
components of $Hyp (D)$.
\medskip

\noindent\begin{proof}[Proof of Theorem \ref{teo-pol-hiper-impares}]
{\bf Case $i$)} 
In order to verify that the Hessian polynomial of $g_{2n+2}$ is negative outside the origin,
we shall prove it first for the straight lines having a slope greater than or equal to 1. To carry out this
case, we will prove that the restriction of the Hessian polynomial of $g_{2n+2}$ to the line $y=1$ 
is negative when $0 \leq x \leq 1$. Since the $x$-variable polynomial obtained consists only of monomials of 
even degree, the proof is also valid for $-1 \leq x\leq 0$, that is, for the straight lines having a negative slope of magnitude greater than or equal to $1$. 
This proof will imply the other cases. The function, obtained as the restriction of the Hessian polynomial of $g_{2n+2}$ to the line $x=1$, limited to the interval $|y| \leq 1,$
corresponds to the study of the Hessian polynomial of $g_{2n+2}$ on the rest of the straight lines.
The symmetries of $g_{2n+2}$, like $g_{2n+2}(x,y) = - g_{2n+2}(y,x)$ imply 
symmetries of $\Hess g_{2n+2}$, like $(\Hess g_{2n+2})(x,y) = (\Hess g_{2n+2})(y,x)$. Because of this, if $(\Hess g_{2n+2})(x,y)<0$ for all $(x,y)$ in the straight line joining the points $(-1,1)$ and 
$(1,1)$ implies that $(\Hess g_{2n+2})(x,y)<0$ for all $(x,y)$ in the line that joins $(1,-1)$ with 
$(1,1)$.
\medskip

After some straightforward calculations we obtain:
\medskip

\noindent $(\Hess g_{2n+2})(x,1) = h_0(x)+h_2(x)+h_{2n-2}(x)+h_{2n}(x)+h_{2n+2}(x)
+h_{4n-2}(x)+h_{4n}(x), $
\medskip

\noindent where: 

\begin{align*}
 &h_0(x)=-4-12n-8n^2 \\
 &h_2(x)=(-4n-8n^2)x^2 \\
 &h_{2n-2}(x)=(-4n-4n^2 +16n^3 + 16n^4) x^{2n-2} \\
 &h_{2n}(x)=(-8-24n-24n^2 -32n^3-32n^4) x^{2n} \\
 &h_{2n+2}(x) = (-4n-4n^2 +16n^3+16n^4) x^{2n+2} \\
 &h_{4n-2}(x)=(-4n-8n^2) x^{4n-2} \\
 &h_{4n}(x)=(-4-12n-8n^2) x^{4n}. 
\end{align*}

In order to prove that $(\Hess g_{2n+2})(x,1)$ is negative for $x\in [0,1],$ we shall consider the polynomial 
$S= h_0+h_2+h_{2n-2}+h_{2n}+h_{2n+2}$ that satisfies $(\Hess g_{2n+2})(x,1) \leq S(x)$ for $x\in \R$.
For $n \leq 11$, it can be checked with a computer that $S(x) \leq 0$ for $x\in [0,1]$.
\medskip

For $n \geq 11$, note that 
\[ S(x) \leq \big(-8 n^2 - 8 n^2 x^2 +16 n^4 (1-x^2)^2x^{2(n-1)}
\big) + \big(-12 n -4n x^2 + 16n^3   (1-x^2)^2x^{2(n-1)} \big).\] 
It only remains to prove that the last sum is negative for $x\in [0,1]$.
After a straightforward calculus, we have the following result.
\begin{lemma}\label{lema-1} Let $g: \left[0,1\right] \rightarrow \mathbb{R}$ be the function
$ g(x)=(1-x^2)^2x^{2(n-1)}$. This function reaches a maximum value at the point 
$x_n=\sqrt{\frac{n-1}{n+1}}$, and it is $\, g(x_n)= \frac{4}{(n+1)^2 \left(1+\frac{2}{n-1}  \right)^{n-1}}.$
\end{lemma}

\begin{lemma}\label{lema-2}
For any $n \geq 11$ and for $0 \leq x \leq 1$ the inequality
\begin{align}\label{des-1}
-8n^2-8n^2x^2+16n^4g(x) <0
\end{align}
holds. 
\end{lemma}

\noindent {\bf Proof.}
\begin{enumerate}
\item 
For $n \geq 11$, it can be shown using Lemma \ref{lema-1} that
\begin{align}\label{g-valor-max}
g(x_n) = \frac{4}{(n+1)^2} \left( \frac{1}{a_{n-1}}\right) \leq \frac{4}{5} \frac{1}{(n+1)^2}.
\end{align}
This inequality implies,
\[ 16n^4g(x) \leq 16n^4g(x_n) \leq 16 \left( \frac{4}{5}\right) \left(\frac{n^2}{(n+1)^2}\right) (n^2) 
< 13 n^2. \]
Thus,
\[ -8n^2 +16n^4g(x) -8n^2x^2 < -8n^2+13n^2 = 5n^2 -8n^2x^2,\] and this expression 
 is less than zero for $x \in \left( \sqrt{\frac{5}{8}}, 1 \right]$.
\item Consider now the interval $\left[0, \sqrt{\frac{5}{8}} \right]$ and $n \geq 11$. So, the function 
$g$ is increasing in $\left[0, \sqrt{\frac{5}{8}} \right]$. It can be shown that
\[ 16 n^4 g(x) \leq n^4 \left(\frac{9}{4}\right) \left(\frac{5}{8} \right)^{n-1} < 300. \]
Finally,  \ 
$-8n^2x^2-8n^2 +16n^4g(x) < -8n^2 +16n^4g(x) < -8n^2+300 <0. \ \ $ $\scalebox{0.8}{$\blacksquare$}$ 
\end{enumerate}
\vskip 0.3cm

\begin{lemma}\label{lema-3}
For each $n \geq 11$ and for $0 \leq x \leq 1$ the inequality
\begin{equation}\label{des-2}
-12n-4nx^2+16n^3g(x) <0
\end{equation}
holds. 
\end{lemma}

\noindent \begin{proof}
As a consequence of inequality (\ref{g-valor-max}) we get the inequalities
\[ 16n^3g(x) \leq 16n^3g(x_n) \leq 16 \left( \frac{4}{5}\right) \left(\frac{n^2}{(n+1)^2}\right) (n) < 13 n. \]
Thus, for $x \in \left( \frac{1}{2}, 1 \right]$ the inequality
$\, -12n-4nx^2+16n^3g(x) <0 \,$ holds.
 Consider now the interval $\left[0, \frac{1}{2} \right]$, and $n \geq 11$. Note that the function $g$ 
is increasing in $\left[0, \frac{1}{2} \right]$.
Moreover, for $n\geq 5$ it is true that $16n^3g(x)\leq 16n^3g(1/2) < 5$. Therefore,

\bigskip 
\hskip 3cm $-12n-4nx^2+16n^3g(x) < -12n-4nx^2+5 < 0. $
\end{proof}
\medskip

Finally, by Lemma \ref{lema-2} and Lemma \ref{lema-3} we conclude that 
$\, (\Hess g_{2n+2})(x,1) < 0 \mbox{ for } x \in \left[0, 1 \right].$
\medskip

\noindent {\bf Case $ii$)} The proof is based on the inductive method on $k\geq 1$. 
\begin{itemize}
\item Taking $k=1$, we shall prove that the polynomial 
$f_{3}(x,y) := x\, (x^{2n} + y^{2n})  (x^2 - y^2)$ is hyperbolic. Consider the following result.

\begin{lemma}\label{lema-uno-mas-lineal}
Let $f\in H_D[x,y]$ be a $D$-hyperbolic polynomial and $\ell (x,y) = ax+by$ be a linear 
polynomial. Then,  the
product $\ell f$ is a $(D+1)$-hyperbolic polynomial if and only if $\ell$ is not a 
factor of $\ell_x f_y - \ell_y f_x$.
\end{lemma}

\noindent \begin{proof}
The hessian polynomial of the product $\,\ell f$ is  \medskip

{\small
{$\Hess$($\ell f$)($x,y$) $=  
\ell^2\Hess f  + 2\ell \ell_y (f_y f_{xx} - f_x f_{xy}) + 2\ell \ell_x (f_x f_{yy} - f_y 
f_{xy}) + 2\ell_x \ell_y f_x f_y - \ell_x^2 f_y^2 - \ell_y^2 f_x^2.$}}
\medskip

\noindent The well known {\bf Euler's formula} for a homogeneous polynomial of degree $D$ in $H[x,y]$ asserts
\begin{eqnarray}\label{formula-euler}
D f(x,y) = x f_x(x,y) + y f_y (x,y). 
\end{eqnarray}
Using it for the homogeneous 
polynomials $f_y$ and $f_x$, 
we get that the expression $f_y f_{xx} - f_x f_{xy}$ becomes
\begin{align*}
(f_y f_{xx} - f_x f_{xy})(x,y) &= \frac{1}{D-1}\left[ x f_{xy}f_{xx} + y f_{yy}f_{xx} - 
x f_{xx} f_{xy} - 
y f_{xy}^2 \right] \\
&= \frac{1}{D-1} y (\Hess f)(x,y),
\end{align*} 
and, in a similar way, $f_x f_{yy} - f_y f_{xy} = \frac{1}{D-1} x (\Hess f)(x,y).$ Replacing them in 
$\Hess (\ell f)$, we get
\medskip

\centerline{$\Hess$($\ell f$)($x,y$) $= \left(\frac{D+1}{D-1}\right) \ell^2 (\Hess f)(x,y) - (a f_y - b f_x)^2.$}
\medskip

\noindent Finally, $\Hess (\ell f)$ is negative in $\R^2\setminus\{(0,0)\}$ if and only if $\ell$ 
is not a factor of the polynomial $\ell_x f_y - \ell_y f_x$.
\end{proof}
\medskip

We now use Lemma \ref{lema-uno-mas-lineal} for the polynomials $g_{2n+2}$ and $\ell (x,y) = x$. 
Denote $G := g_{2n+2}$. We 
claim that $x$ is not a factor of the polynomial $(\ell_x G_y - \ell_y G_x)(x,y) = G_y (x,y) = 
(x^2-y^2)(2n y^{2n-1}) -2y (x^{2n} + y^{2n})$. Indeed, since $(G_y |_{y=1}) (0) = -2n-2 <0$,
we conclude that $x=0$ is not a root of $(G_y |_{y=1}) (x)$. 
Therefore, the polynomial $x g_{2n+2}(x,y)$, which is $f_{3}$, is hyperbolic.
\medskip

\item Inductive hypothesis. Let suppose now that $\ f_{2k+1} $ is hyperbolic.

\item We shall now verify that $\ f_{2k+3}(x,y) := x\, (x^{2n} + y^{2n}) 
\Pi_{i=1}^{k+1} (x^2 - i^2 y^2)$ is hyperbolic. The proof is divided into two steps: in the 
first one we shall demonstrate that $ (x - (k+1) y)f_{2k+1}(x,y)$ is 
hyperbolic; and in the second one we shall prove that  $ (x + (k+1) y) (x - (k+1) y)f_{2k+1}(x,y)$, 
which is $ f_{2k+3}(x,y)$, is hyperbolic.
\end{itemize}

\begin{lemma}\label{lema-pol-hyp-grado-par}
The even degree homogeneous polynomial $ (x - (k+1) y) f_{2k+1}(x,y)$ is hyperbolic 
provided that $ f_{2k+1}(x,y)$ is hyperbolic.
\end{lemma}

\noindent\begin{proof}
Let us denote $g(x,y) := \Pi_{i=1}^k (x^2 - i^2 y^2), \,$ $f := f_{2k+1}$  and 
$\ell (x,y) =  (x - (k+1) y)$.
In order to verify that $\ell f$ is hyperbolic, it is enough to prove, according to 
Lemma \ref{lema-uno-mas-lineal}, 
that $\ell$ is not a factor of  $\ell_x f_y - \ell_y f_x$.
Since $ f= x \left(x^{2n} + y^{2n}\right) g(x,y), $ we have that
\begin{eqnarray}\label{derivadas-f}
f_x (x,y) =  \left(x^{2n} + y^{2n}\right)\big( x g_x + g\big) +  2n g x^{2n}, 
\ \ \ 
f_y (x,y) =  \left(x^{2n} + y^{2n}\right)\big( x g_y\big) +  2n x g y^{2n-1}.
\end{eqnarray}
Replacing the expressions of (\ref{derivadas-f}) in $\ell_x f_y - \ell_y f_x$ we obtain
{\small
\begin{align*}
\ell_x f_y - \ell_y f_x &= \left(x^{2n} + y^{2n}\right)\big( x g_y \big) 
+ 2n x y^{2n-1} g + (k+1)
\left(x^{2n} + y^{2n}\right)\big( x g_x + g \big) + 2n  (k+1) g x^{2n}\\
&= \left(x^{2n} + y^{2n}\right)\Big( x g_y + (k+1) x g_x + (k+1) g \Big) 
+ 2 n g \big( x y^{2n-1} + (k+1)x^{2n}\big).
\end{align*}}
Using Euler's Lemma (see equation (\ref{formula-euler})) for the homogeneous 
polynomial $g$, we get the relation
$ x g_x  = 2k g - y g_y$ which is placed in $\ell_x f_y - \ell_y f_x $.
{\small
\begin{align*}
\ell_x f_y - \ell_y f_x &=
\left(x^{2n} + y^{2n}\right)\Big( x g_y + (k+1) 2k g - (k+1)y g_y + (k+1) g \Big) 
+ 2 n g \big( x y^{2n-1} + (k+1)x^{2n}\big)\\
&= \left(x^{2n} + y^{2n}\right)\Big( (k+1) (2k +1) g + \big(x - (k+1)y\big) g_y \Big) 
+ 2 n g \big( x y^{2n-1} + (k+1)x^{2n}\big).
\end{align*}}
To prove that $\ell$ is not a factor of  $\ell_x f_y - \ell_y f_x$ it will be enough 
to make sure that $x=k+1$ is not a root of the one-variable polynomial  
$F(x) := (\ell_x f_y - \ell_y f_x)|_{y=1}$. To carry it out, we shall
verify that $F$ is positive in $x= k+1$.
\begin{align*}
F(x) = &\left(x^{2n} + 1\right)\Big( (k+1) (2k +1) \Pi_{i=1}^k (x^2 - i^2) + 
\big(x - (k+1)\big) g_y|_{y=1}\Big) \\
&+  2n \big( x + (k+1)x^{2n}\big) \Pi_{i=1}^k (x^2 - i^2),
\end{align*}
Hence, 
\begin{align*}
F(k+1) =
&\big((k+1)^{2n} + 1\big)\Big( (k+1) (2k +1) \Pi_{i=1}^k \big((k+1)^2 - i^2\big) 
\Big) \\
&+  2n \big( k+1 + (k+1)^{2n+1}\big) \Pi_{i=1}^k ((k+1)^2 - i^2).
\end{align*}
Since $1\leq i \leq k$, then $i^2 < (k+1)^2$. So, the product $\,\Pi_{i=1}^k 
\big((k+1)^2 - i^2\big) $ is a 
positive number, and therefore the number $F(k+1) $ is positive.
\end{proof}
\medskip

We now shall prove that the polynomial $ f_{2k+3}$ is hyperbolic. Using the same 
notation of Lemma \ref{lema-pol-hyp-grado-par}, 
$g(x,y) := \Pi_{i=1}^k (x^2 - i^2 y^2), \,$ $f := f_{2k+1}$,\,
$\ell (x,y) =  x - (k+1) y$, and denoting $P = \ell f$ and $L(x,y) = x + (k+1)y$,
we have just to verify, according to Lemma \ref{lema-uno-mas-lineal},  
that $L$ is not a factor of  $L_x P_y - L_y P_x$. Since
$$P_x (x,y) = \ell f_x + f \ \mbox{ and }\ P_y (x,y) = \ell f_y - (k+1) f,$$
then
\begin{align*}
L_x P_y - L_y P_x &= (x - (k+1)y) f_y - (k+1) f  - (k+1) \Big( f + (x - (k+1)y)  
f_x\Big) \\
&= - 2 (k+1) f +  \Big(x - (k+1) y\Big) \Big( f_y - (k+1) f_x\Big).
\end{align*}
Let us analyze separately the addends of the last expression:
{\small 
\begin{align*}
 f_y - (k+1) f_x &=
\left(x^{2n} + y^{2n}\right)\big( x g_y \big) + 2n x y^{2n-1} g - (k+1)
\left(x^{2n} + y^{2n}\right)\big( x g_x + g \big) - 2n  (k+1) g x^{2n}\\
&= \left(x^{2n} + y^{2n}\right)\Big( x g_y - (k+1) x g_x - (k+1) g \Big) 
+ 2 n g \big( x y^{2n-1} - (k+1) x^{2n} \big).
\end{align*}}
Since $g$ is a $2k$-degree homogeneous polynomial, we can use the Lemma of Euler 
(see equation (\ref{formula-euler})) for it, obtaining the relation 
$  (k+1) x g_x  = 2k  (k+1) g -  (k+1) y g_y$. So,
{\small 
\begin{align*}
 f_y - (k+1) f_x &=
\left(x^{2n} + y^{2n}\right)\Big( x g_y - 2k (k+1) g + (k+1)y g_y - (k+1) g \Big) 
+ 2 n g \big( x y^{2n-1} - (k+1) x^{2n} \big)
 \\
&= \left(x^{2n} + y^{2n}\right)\Big(  \big(x + (k+1) y\big) g_y -(k+1) (2k +1) g  
\Big) + 2 n g \big( x y^{2n-1} - (k+1) x^{2n} \big).
\end{align*}}
Denote $R(x) := (L_x P_y - L_y P_x)|_{y=1}$. So, if $x=-(k+1)$ is not a root of $R$, 
then $L$ is not a factor 
of $L_x P_y - L_y P_x$. Let us note that
\begin{align*}
R(x) = & - 2 (k+1) x\, (x^{2n} + 1) g|_{y=1} +  \Big(x - (k+1)\Big) 
\Big(x^{2n} + 1\Big)\Big(  \big(x + (k+1)\big) g_y|_{y=1} \Big.\\
&\Big. -(k+1) (2k +1) g|_{y=1}  \Big) + 2 n \Big(x - (k+1)\Big) 
\big( x - (k+1) x^{2n} \big) g|_{y=1}.
\end{align*}
Therefore,
\begin{align*}
R(-(k+1)) & = (k+1)^{2} \Big((k+1)^{2n} + 1\Big) \Pi_{i=1}^k \big((k+1)^2 - i^2\big)\\
&+ \Big( -2(k+1)\Big)\Big((k+1)^{2n} + 1\Big)\Big( -(k+1) (2k +1) 
\Pi_{i=1}^k \big((k+1)^2 - i^2\big) \Big)\\
&+  2 n \Big( - 2(k+1)\Big) 
\big( - (k+1) - (k+1)^{2n+1} \big)\Pi_{i=1}^k \big((k+1)^2 - i^2\big).
\end{align*}
Since $i^2 < (k+1)^2$ for each $1\leq i \leq k$, each addend of $R(-(k+1))$ is positive.
We conclude that $R(-(k+1))$ is positive. \medskip

\noindent {\bf Case $iii$)} It was included in the third step of case $ii$).
\end{proof}
\bigskip

\subsection{Some properties of the graph of the hyperbolic polynomials of 
Remark \ref{pol-fact} and Theorem \ref{teo-pol-hiper-impares}}

\begin{itemize}
\item  Let $f$ be a hyperbolic polynomial in the family mentioned above, that is, $f = E R,$ where $R$ is the product of non-proportional linear homogeneous polynomials and $E$ denotes some polynomial $E_{2n}$. The function $F$, which is the 
restriction of $f$ to the unit circle, $F ( \varphi ) = f (cos  \,\varphi, sin \, \varphi )$, satisfies, according to 
Remark \ref{obs-arnold}, that between each 
two of its consecutive zeros there is only one critical point. Moreover, according to inequality 
(\ref{hyper-cond}), its critical points are not degenerate. By the homogeneity of $f$, the same 
behavior occurs on circles of radius $r$. In Figure 1, we show the graph of 
$f(x,y) = (x^2-y^2)(x^2- 4y^2)(x^4 +y^4)$, and in Figure 2 we show the graph of $F$ 
on the interval $[0,\pi]$. Both graphs were done with the Maple software.
\medskip

\begin{figure}[ht] 
\begin{center}
 \begin{minipage}{0.5\textwidth}
\qquad\includegraphics[width=6.7 cm, height=5.1 cm]{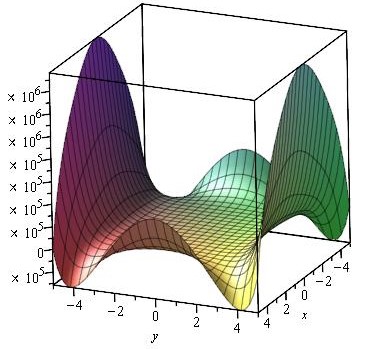}
\caption{Graph of the polynomial \\ $f(x,y) = (x^2-y^2)(x^2- 4y^2)(x^4 +y^4)$.}
\end{minipage}
\quad
\begin{minipage}{0.4\textwidth}
 \includegraphics[width=6.5 cm, height=5 cm]{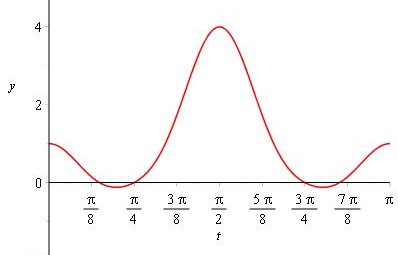}
\caption{Graph of $F$ on $[0,\pi]$.}
\qquad
\end{minipage} 
\end{center}
\end{figure}

\item In this item, we shall prove the following result.
\begin{proposition}
The foliations of the asymptotic curves of the surface $z = f (x, y)$, and those of the surface $z = R (x, y)$, 
are topologically equivalent, where $f (x, y) = E(x,y) R(x,y),\,$ $E(x,y)=x^{2m}+y^{2m}$, and $R(x,y)$ is  
the product of non-proportional real linear homogeneous factors (i.e., $R$ is of the form of one of those in Remark \ref{pol-fact}). 
\end{proposition}

To do this,
we will study the fields of asymptotic directions of (the surfaces determined by) the polynomials.
In \cite{O-Sa}, the authors prove in Theorem 1.1 that the second fundamental form of the surface 
$z-f_m (x,y)=0,$ where $f_m$ is a homogeneous polynomial of degree $D$ in the Arnold
family (\ref{polis-arnold}), is hyperbolic isotopic to the second fundamental form of the surface
$z = \mbox{Re} (x+i y)^m$ if $m\geq 2, m\leq D\leq m^2$, and $D-m$ is even. In this case, 
a quadratic form on 
$\mathbb R^{2^*}$ is called {\it smooth hyperbolic} if its discriminant is positive
at each point on $\mathbb R^{2^*} = \R^ 2\setminus \{(0,0)\}$, and is extended 
to the origin with a singularity. We recall that {\it the discriminant of a quadratic form $ a dx^2 + 2 b dx dy + c dy^2$} is $\, b^2- ac.$
And a {\it hyperbolic isotopy} between two smooth hyperbolic quadratic forms $\omega$ and 
$\delta$ on $\mathbb R^{2^*}$ is a smooth map
$$\Psi: \mathbb R^{2^*} \times [0,1]  \rightarrow {\cal Q},\ \ \ (x,y,t) \mapsto \Psi_t(x,y),$$
\noindent where ${\cal Q}$ is the space of real quadratic forms on the plane and the following conditions hold:
$\Psi_0(x,y)= \omega(x,y)$, $\Psi_1(x,y)= \delta(x,y)$ and $\Psi_t(x,y)$ 
is a smooth hyperbolic quadratic form on 
$\mathbb R^{2^*}$, which extends at the origin with a singularity. In such a case, it is said that 
{\it $\omega$ and $\delta$ are hyperbolic isotopic}. If the second fundamental forms of two 
hyperbolic homogeneous polynomials are hyperbolic isotopic, then the indices of their fields 
of asymptotic lines at the origin coincide. Moreover, 
the asymptotic fields are locally topologically 
equivalent; this means that the number of separatrices of both fields are the same, and the corresponding
sectors are of the same type (parabolic, hyperbolic or ellyptic).

\medskip

We remark that the arguments used in the proof of Theorem 1.1 of
\cite{O-Sa} are also valid for the hyperbolic polynomials of Theorem \ref{teo-pol-hiper-impares}.
For completeness, we give a  proof of this statement. In what follows consider the notation 
$Q(x,y):=x^{2n}+y^{2n}$, and denote by $P(x,y)$ 
the product of $k$ non-proportional linear homogeneous polynomials. Then, the second fundamental form of 
the polynomial $PQ$ of degree $k+2n$ is 
$$\II_{PQ} = P\II_Q + 2dP dQ + Q\II_P,$$
where $dP= P_x dx + P_y dy$, $dQ=Q_xdx + Q_y dy$ and $dPdQ$ is the quadratic form defined 
by the product of $dP$ and $dQ$. \medskip

\noindent 
$i$) We claim that the quadratic form $\omega :=  2dP dQ + Q\II_P$ is a smooth
hyperbolic quadratic form on $\mathbb R^{2^*}$.  Indeed,
after some computations, we have that the discriminant of $\omega$ is
$$\Delta_\omega = -Q^2 \Hess P + (P_x Q_y -P_y Q_x)^2 - 2 Q T,$$
where $T := P_{xx}P_y Q_y + P_{yy} P_x Q_x -P_{xy}(P_x Q_y + P_y Q_x).$
In accordance with Remark 3.2 of \cite{O-Sa}, the polynomial $T$ becomes, after using 
Euler's Lemma, the expression $ \left(\frac{2n}{k-1}\right) Q \Hess P$, which is 
negative in $\mathbb R^{2^*}$. Thus, the 
discriminant $\Delta_\omega$ is positive in $\mathbb R^{2^*}$.\medskip

\noindent 
$ii$) Consider the quadratic form $\delta := P II_Q$. Then, the hyperbolic quadratic forms 
$\omega$ and $\omega + \delta = II_f$ are hyperbolic isotopic. In order to prove this relation, 
we denote $\omega (x,y) = \omega_1 dx^2
+ 2\omega_2 dx dy + \omega_3 dy^2,$  $\, \delta (x,y) = \delta_1 dx^2 + 2\delta_2 dx dy + 
\delta_3 dy^2,$ and consider  the isotopy:
$$ \Phi_t (x,y) = \omega(x,y) + t \delta(x,y), \mbox{ with } t\in [0,1]$$ 
whose discriminant is:
\[\Delta_{\Phi_t} = \omega_2^2 - \omega_1 \omega_3 + t \, (2\omega_2 \delta_2 - \omega_1\delta_3- \omega_3 \delta_1) + t^2 \, ( \delta_2^2 - \delta_1 \delta_3).\]
We know that $\,\omega_2^2 - \omega_1 \omega_3 >0$  in $\mathbb R^{2^*}$ because of $\omega$ 
is hyperbolic; the term $\,\omega_2^2 - \omega_1 \omega_3 + (2\omega_2 \delta_2- \omega_1\delta_3- \omega_3 \delta_1) + ( \delta_2^2 - \delta_1 \delta_3)$ is positive in $\mathbb R^{2^*}$  because of 
$\omega + \delta$ is hyperbolic, and $\Delta_\delta = \delta_2^2 - \delta_1 \delta_3 = - P^2 \Hess Q \leq 0$ in $\mathbb R^{2^*}$ since $(\Hess Q)(x,y) = 4 n^2 (2n-1)^2 x^{2n-2} y^{2n-2} \geq 0$.
From all these inequalities, we conclude that the discriminant $\Delta_{\Phi_t}$
is positive in $\mathbb R^{2^*}$ for any $t\in [0,1]$ because for each $(x,y)$ in $\mathbb R^{2^*},$ 
it is a polynomial $q(t) = c +bt + at^2$ such that $q(0) >0, \, q(1) >0$ and $a\leq 0$.
\medskip

\noindent 
$iii$) We assert that the forms $\omega$ and $Q\II_P$ are hyperbolic isotopic. The isotopy that  fulfills
this property is
$$\Psi_t (x,y)= Q II_p + 2t dP dQ, \mbox{ with } \,t \in [0,1]$$
because its discriminant 
\[ \Delta_{\Psi_t} = - Q^2 \Hess P  - 2t Q T + 4 t^2 \Delta_{dP dQ}\]
is positive in $\mathbb R^{2^*}$ for any $t\in [0,1]$ because $-T >0$ and $\Delta_{dP dQ} = 
\frac{1}{4}(P_x Q_y - P_y Q_x)^2 \geq 0$  in $\mathbb R^{2^*}$.
\medskip

\noindent 
$iv$) The quadratic forms $Q\II_P$ and $\II_P$ are hyperbolic isotopic. 
Indeed, the isotopy that carries it out is $\Gamma_t (x,y):= t\II_P + (1-t) Q\II_P$ with $\,t \in [0,1]$
since its discriminant $\,\Delta_{\Gamma_t} = -M^2 \Hess P$, with $M = t + (1-t)Q$, is positive in
$\mathbb R^{2^*}$.
\medskip

\noindent Finally, the fields of asympotic lines corresponding to the polynomials $f = QR$ and $R$ are topologically equivalent. $\scalebox{0.8}{$\blacksquare$}$
\medskip

\item In this item, we shall describe the topological behavior of the asymptotic curves of a
homogeneous polynomial $P$ that is the product of $k$ non-proportional real linear homogeneous 
polynomials. Namely,
$$P(x,y)= \Pi_{i=1}^k \ell_i (x,y) \ \mbox{ where } \ \ell_i (x,y) := a_ix+b_iy, \ i\in \{1, \ldots ,k\}.$$

\begin{proposition}\label{prop-sectores-hip}
The second fundamental form of the surface $S \subset \mathbb{R}^3$ defined by the equation $z=P(x,y)$ defines globally two fields of asymptotic directions on the $xy$-plane. Moreover, each of them satisfies the following properties:
\begin{enumerate}
\item[$i$)]   Has only one singular point, which is the origin.
\item[$ii$)]  There are $k$ rays that are asymptotic lines, all of them, emanating from the origin.
\item[$iii$)]  Each sector between two adjacent rays is hyperbolic.
\end{enumerate}
\end{proposition}

In Figure \ref{pol-cubico} and Figure \ref{pol-cuartico}, we exhibit respectively,  
the asymptotic curves of a field of asymptotic 
directions for the polynomials $f(x,y) = x (x^2-y^2)$ and $f(x,y)=xy(x^2-y^2)$.
\begin{figure}[ht] 
\begin{center} 
\begin{minipage}{0.45\textwidth}
\qquad\qquad \includegraphics[width=3.6 cm, height=3.0 cm]{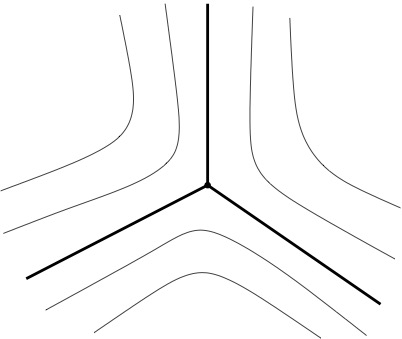}
\caption{Foliation of a Field of Asymptotic Directions for $f(x,y)=x(x^2-y^2)$.}
\label{pol-cubico} 
\end{minipage}\qquad
\begin{minipage}{0.45\textwidth}
\qquad\qquad \includegraphics[width=3.6 cm, height=3.0 cm]{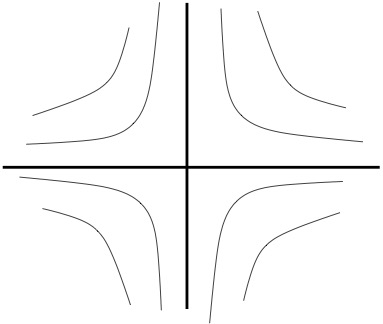}
\caption{Foliation of a Field of Asymptotic Directions for $f(x,y)=xy(x^2-y^2)$.}
\label{pol-cuartico}
\end{minipage} 
\end{center}
\end{figure}

\noindent{\bf Proof.}
\begin{itemize}
\item[$i$)] Since the surface $S: \ z-P(x,y)=0$ is orientable, its second fundamental form defines globally two fields of asymptotic directions on the $xy$ plane. When the degree of $P$ is greater than 2, both fields have a singular point at the origin, and it is the only one since the other points are hyperbolic.

\item[$ii$)] Consider the half-lines obtained by removing the origin from the straight lines 
$\ell_i (x,y) = 0, \ i\in \{1, \ldots ,k\}.$ We claim that each of these half-lines is an integral curve of one of the fields of asymptotic lines. In order to prove this assertion, consider a point $q$ in the line 
$\ell_i (x,y) = 0$ and let $v$ a vector tangent to this line at $q$. We shall verify that the vector $v$ lies on 
the zero locus of the fundamental form $\II_P (q).$ 
Note that the point $q$ has coordinates $\lambda (\bar x, \bar y)$, where $\,\bar x = -b_i, \, \bar y = a_i,$ 
and $\,\lambda \neq 0$; and let us take as the vector $v$, the vector $(\bar x, \bar y)$. By evaluating the form $\II_P (q)$ on vector $v$, 
and due to the homogeneity property of $P$, we obtain that $\II_P (q) (v) = 
\lambda^{k-2} \omega$, where
\[\omega = {\bar x}^2 P_{xx}(\bar x, \bar y) + 2 {\bar x}{\bar y} P_{xy}(\bar x, \bar y) + 
{\bar y}^2 P_{yy}(\bar x, \bar y).\]
Using Euler's Formula (\ref{formula-euler}) for the polynomials $P, \ P_x$ and $P_y$, we infer that 
\[k P (\bar x, \bar y) = \bar x \left( \frac{1}{k-1} \Big(\bar x P_{xx}  (\bar x, \bar y) 
+ \bar y P_{xy} (\bar x, \bar y)\Big)\right) + \bar y \left( \frac{1}{k-1} \Big(\bar x P_{xy}  (\bar x, \bar y) 
+ \bar y P_{yy} (\bar x, \bar y)\Big)\right).\]
This implies that $\omega =  k (k-1) P (\bar x, \bar y)$. Since $P (\bar x, \bar y) = 0$, we conclude that the quadratic form $\II_P (q)$ vanishes at the vector $v$.
\medskip

We claim that two consecutive half-lines are integral curves of different fields
of asymptotic lines. In order to prove this, consider appropriate real numbers $\epsilon, c_1, \cdots ,c_k$  
such that the non-homogeneous polynomial 
$$E(x,y) = \Pi_{i=1}^k E_i (x,y)  \mbox{ \ where \ } 
E_i (x,y) := c_i \epsilon + a_i x+ b_i y,   \  i \in \{1, \ldots ,k\},$$  
is a {\it factorizable polynomial} (see \cite[Definition 2]{AOR}), that is, the straight lines determined by the equation $E (x,y) = 0$ 
are in generic position. Let us denote such straight lines respectively as $L_1, \ldots ,L_k$.
\medskip 

In \cite[Theorem 1]{AOR}, the author studies the geometrical behavior of these kinds of 
polynomials: each straight line $L_i$, $i \in \{1, \ldots ,k\}$ contains exactly $k-2$ special parabolic points
(points at which their only asymptotic direction is tangent to the smooth parabolic curve), which divide $L_i$ into $k-1$ open segments. Each of these segments is an integral curve of one of the fields of asymptotic lines. Furthermore, any two adjacent segments belong to different line fields.
So, if the degree of $E$ is even, the two unbounded segments of $L_i$ are integral curves of the same direction field whereas if the degree is odd, the unbounded segments belong to different direction fields. 
The position of the $k-2$ special parabolic points on $L_i$ is as follows. The line $L_i$ is divided by the other 
$k-1$ lines into $k$ segments of which $k-2$ are bounded. In each open bounded segment there is exactly one special parabolic point.
\medskip

On the other hand, the lines $L_1, \ldots, L_k$ determine several compact polygons on the $xy$-plane. 
Let $M$ denote the union of all these compact polygons. Thus, all the special parabolic points are in $M$.
By letting $\epsilon$ now tend to zero, the polynomial $E(x,y)$ tends to the polynomial $P(x,y)$, and by continuity, the compact polygon $M$ tends to the origin (because all intersection points $L_i \cap L_j$ 
tend to the origin). Thus, any two consecutive half-lines of $P(x,y)=0$ are integral curves of different fields
of asymptotic lines. Moreover, if the degree of $P$ is even the two half-lines of $L_i$ are integral curves of 
the same field, whereas if the degree of $P$ is odd, they are in different direction fields. 
\medskip

\item[$iii$)] In this item, we shall prove that each open sector on the $xy$-plane bounded by two adjacent half-lines (defined in $ii$)) is of a hyperbolic type.
In \cite{GuOr}, the authors study the extension to the real projective plane of the fields 
of asymptotic lines of a real polynomial $f$ (not necessarily homogeneous) in two variables. 
They prove that 
($a$) there are two smooth fields of directions over the unit sphere such that if we restrict them to an open hemisphere, we obtain two fields of lines diffeomorphic to the fields of asymptotic directions. In this case, 
the equator plays the role of the line at infinity. These fields over the sphere are called extended direction fields;
($b$) the singular points on the equator of the extended direction fields (called singular points at infinity)
are determined by the leading homogeneous part of the polynomial of $f$;
($c$) the equator without the singular points at infinity is an integral curve of both extended direction fields;
($d$) the Poincar\'e index of each singular point at infinity is $\frac{1}{2}$, and
($e$) the topological type of such a singular point at infinity is like the type of a Monstar, 
which is shown in Figure \ref{sing-pt-infinity}. In Figure \ref{sing-pts-esfera}, the behavior of a field of 
asymptotic directions in a neighborhood 
of a singular point at infinity is exhibited according to the parity of the degree of $f$ \cite[Remark 4.10]{GuOr}.

\begin{figure}[ht] 
\hskip 2cm 
\begin{minipage}{0.4\textwidth}
\quad \includegraphics[width=4.3 cm, height=3.0 cm]{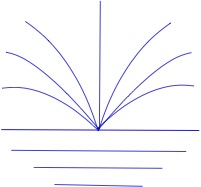}
\caption{Topological type of a \\ singular point at infinity}
\label{sing-pt-infinity} 
\end{minipage}
\quad
\begin{minipage}{0.5\textwidth}
 \includegraphics[width=6.9 cm, height=3.5 cm]{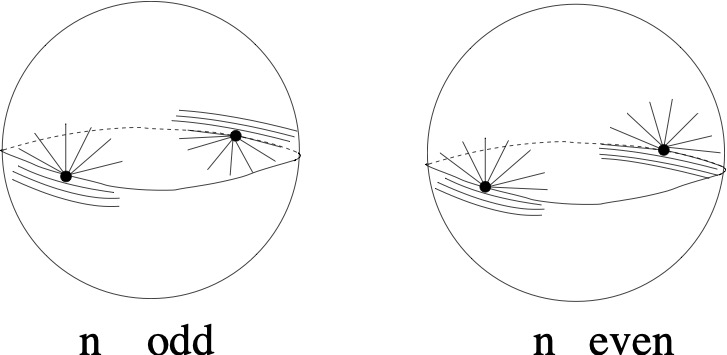}
\caption{Behavior of the asymptotic \\ curves at a singular point at infinity}
\label{sing-pts-esfera} 
\end{minipage} 
\end{figure}

Using all these results for our factorizable polynomial $P$ of degree $k$, we conclude that the extended
direction fields have $2k$ singular points  on the equator (each linear factor of $P$ determines exactly
two points, $q$ and $-q$). In fact, each extended direction field has exactly $k$ singular points at infinity 
(on the equator). Let $F_1$ and $ F_2$ denote the extended direction fields. According to item ($e$), 
we have 
that if $q$ is a singular point at infinity of $F_1$, then $-q$ is a singular point at infinity of $F_1$ or $F_2$ 
whenever $k$ is, respectively, even or odd. Let us now restrict one of the extended direction fields to the 
closed northern hemisphere (the dynamic over the closed southern hemisphere is similar). 
This restricted line field 
satisfies the following properties: it has the north pole as its only singular point outside the equator, it has $k$ singular points at infinity, and $k$ separatrices that go from the north pole to the $k$ singular points at infinity. Thus, the closed northern hemisphere is divided by the $k$ separatrices into $k$ sectors, each bounded by two
separatrices and a segment of the equator. Let $\Sigma$ denote one of these sectors, and $A$ and $B$ the two singular points at infinity that lie in this sector.
According to the local topological behavior of a singular point at infinity, shown in Figure $\ref{sing-pts-esfera}$, 
and since there are no singular points inside $\Sigma$, all integral curves contained inside $\Sigma$ begin at the singular point $A$ and end at the singular point $B$. This shows that the interior of $\Sigma$ is a hyperbolic sector.
\medskip

All this analysis allows us to show in Figure \ref{pol-cubico-ext} and Figure \ref{pol-cuartico-ext}, respectively, the 
topological behavior of asymptotic curves of one of the extended asymptotic direction fields for a cubic and 
a quartic polynomial. The topological behavior of asymptotic curves for a polynomial of higher degree
is similar.

\begin{figure}[ht] 
\begin{minipage}{0.46\textwidth}
\qquad\qquad \includegraphics[width=3.4 cm, height=3.0 cm]{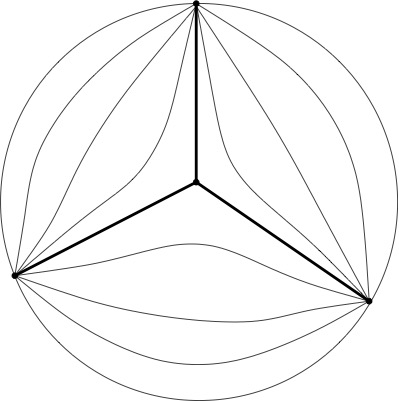}
\caption{Foliation of an Extended Field of Asymptotic Directions for $x(x^2-y^2)$.}
\label{pol-cubico-ext} 
\end{minipage}\qquad
\begin{minipage}{0.46\textwidth}
\qquad\qquad  \includegraphics[width=3.4 cm, height=3.0 cm]{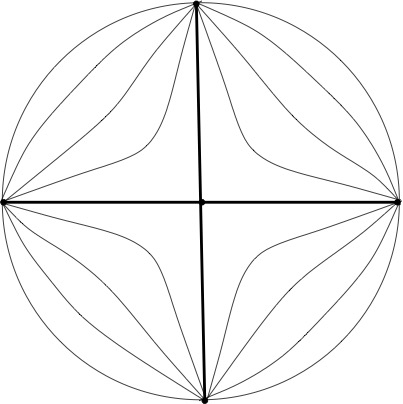}
\caption{Foliation of an Extended Field of Asymptotic Directions for $xy(x^2-y^2)$.}
\label{pol-cuartico-ext}
\end{minipage} 
\end{figure}
\end{itemize}
\end{itemize}

\end{document}